\documentclass{amsart}
\usepackage{color}
\usepackage[utf8]{inputenc}
\usepackage{amscd}
\usepackage{pstricks}
\usepackage{amsmath}
\usepackage[english]{babel}
\usepackage{amsfonts}
\usepackage{graphicx}
\usepackage{vmargin}
\usepackage{theoremref}
\usepackage[hidelinks]{hyperref}
\usepackage{amssymb}
\usepackage{tikz}
\usepackage{relsize}
\usepackage[toc,page]{appendix}
\usepackage{commath}
\usepackage{mathrsfs}
\usetikzlibrary{matrix}
\usepackage{cancel}
\usepackage{lineno}

\newtheorem{theorem}{Theorem}[section]
\newtheorem{lemma}[theorem]{Lemma}

\newtheorem{definition}[theorem]{Definition}

\newtheorem{corollary}[theorem]{Corollary}

\newtheorem{remark}[theorem]{Remark}
 \usepackage{color} 

\usepackage{chngcntr}
\counterwithin{figure}{section}

\numberwithin{equation}{section}

\title[]{SMOOTHNESS  OF  CLASS $C^2$ OF  NONAUTONOMOUS  LINEARIZATION without spectral conditions}
\author[]{N. Jara}

\address{Universidad de Chile, Departamento de Matem\'aticas. Casilla 653, Santiago, Chile}

\email{nestor.jara@ug.uchile.cl}
\thanks{This research has been partially supported by ANID, Beca Nacional de Magister 22200774.}
\date{\today}

\begin{document}

\maketitle

\begin{abstract}
We prove that smoothness of nonautonomous linearization is of class $C^2.$ Our approach admits the existence of stable and unstable manifolds determined by a family of nonautonomous hyperbolicities.  Moreover, our goal is reached without spectral conditions.
\end{abstract}

\section{Introduction}

The linearization problem was initially regarded locally by P. Hartman \cite[Theorem I]{Hartman1} in the context of autonomous equations, later a specific type of those equations, namely a linear and a quasi-linear system, were studied by C. Pugh \cite{Pugh} in order to achieve a global homeomorphism between the flows of the systems. Next this result was adapted to the nonautonomous case by K. J. Palmer \cite{Palmer}, who implemented the concepts of exponential dichotomy in order to obtain its linearization problem. 


The homoeomorphisms with its properties that K.J. Palmer obtained are known as the concept of topological equivalence which is useful to describe the asymptotic behaviour of solutions, even when they may be unknown, for example, to find locally or global attractors, or more generally, stable and unstable manifolds. Nevertheless, there are more dynamical properties that cannot be described with purely topological tools and require a study of the the differentiability of such homeomorphisms.

The study of the smoothness of homomorphisms has been a problem with a vast analysis in both the autonomous and nonautonomous contexts:  In the autonomous framework, S. Sternberg \cite{Sternberg1,Sternberg2} proved that dynamical system given by $C^r-$diffeomorphism can be $C^k-$linearized locally around hyperbolic fixed points that satisfy a nonresonant condition on its spectrum. This results were later improved by G. R. Belicki\u{\i} \cite{Belitski1, Belitski2}, still locally but giving more explicit conditions for the derivatives of the conjugation. S. van Strein \cite{Strien} was the first to obtain a similar result without imposing resonance conditions, but his proof turned out to be wrong, as stated by V. Rayskin \cite{Rayskin}. 

As far as the author has been able to ascertain of studying the differentiability of the topological conjugacy between a linear and a quasilinear nonautonomous system was first regarded in  the work \cite{Castaneda4} by \'A. Casta\~neda and G. Robledo, who ensured the homeomorphisms are $C^2$ preserving orientation diffeomorphisms if the linear system are exponentially  asymptotically stable on $\mathbb{R}$ and verified some conditions. This result was later improved in \cite{Castaneda5} to include more general decay ratios other than exponential but in $\mathbb{R}^+$. In these results authors were able to construct a global diffeomorphism in the same fashion as K.J. Palmer.

L. V. Cuong \textit{et. al.} \cite{Cuong} obtained a smooth linearization in a similar way as S. Sternberg when the linear part admits an exponential dichotomy on whole real line. Recently, D. Dragi\v{c}evi\'c \textit{et. al.} \cite{Dragicevic2} have proved Strein's statement to be true, and further extended it to the nonautonomous case while also improving Cuong's result, under the assumption that the linear system satisfy a nonuniform  hyperbolicity and other technical conditions, which do not include nonresonance conditions but impose spectral bounds.

In this paper we develop ideas hold many similarities with \'A. Casta\~neda, P. Monz\'on and G. Robledo \cite{CMR}, since both study the topological equivalence of a linear and a quasi-linear systems on the positive half real line. We improve their result in two senses: first we present a wider family of dichotomies accepted for the linear system, including both the exponential and the nonuniform exponential under some technical conditions that will be expressed later. Secondly, and  we allow the existence of nonempty unstable manifolds for our linear system. We do not rely on of spectral bounds or nonresonant conditions, but only in the interlacing of the properties of nonautonomous hyperbolicity of the linear part, and boundedness and lipschitzness of the nonlinearities.

However, our work does not regard the continuity of the homeomorphism of topological equivalence as a function of two variables, a property that \'A. Casta\~neda \textit{et al.}  called  continuous topological equivalence \cite[Definition 1.4]{CMR}. We neither obtain a strong topological equivalence, a characterization introduced by J. Shi \cite[Definition 2.4]{Shi}, which corresponds to the uniform continuity of the homeomorphism and was also achieved on \cite[Theorem 2.1]{CMR}. Nevertheless, as a byproduct of the existence of an uniform bound for the derivative of the homeomorphism at any fixed time, we obtain a weaker version of this.

In the third section of this work we present and prove the topological equivalence between our systems, admitting a wide family of dichotomies, similar to those described by P. Gonzalez \textit{et. al.} \cite{Castaneda3} for the discrete framework, work from which we borrow important tools, adapted for the continuous framework. 

In the fourth section we prove under technical conditions that the homeomorphism that we previously constructed is a $C^1-$diffeomorphism and give an explicit form for its derivative. Furthermore, we give examples of concrete dichotomies that satisfy our conditions. Finally, in the fifth section we give enough conditions to ensure the second class of differentiability for the topological equivalence, as well as an explicit way to achieve those conditions with an specific dichotomy. 

\section{Preliminaries}
We study the systems 
\begin{equation}\label{89}
    \dot{x}=A(t)x,
\end{equation}
and
\begin{equation}\label{90}
    \dot{y}=A(t)y+f(t,y).
\end{equation}

Denote $t\mapsto x(t,\tau,\xi)$ y $t\to y(t,\tau,\eta)$ to the solutions of (\ref{89}) and (\ref{90}) that pass through $\xi$ and $\eta$ respectively on $t=\tau$. We also denote $X(t,s)$ the transition matrix of (\ref{89}) such that for $t=s$ is the identity. Moreover, $A:\mathbb{R}^+\to \mathcal{M}_{d}(\mathbb{R})$ is continuous, non singular (\textit{i.e.} has invertible images) and uniformly bounded, that is to say, there exists $M> 1$ such that
    \begin{equation*}
    \max\left\{ \sup_{s \in \mathbb{R}^+}\norm{A(s)}, \sup_{s\in \mathbb{R}^+}\norm{A^{-1}(s)}\right\}=M.
    \end{equation*}
    
    The function $f:\mathbb{R}^+\times \mathbb{R}^d\to \mathbb{R}^d$  is such that there exist sequences $\mathfrak{u}, \mathfrak{v}:\mathbb{R}^+\to \mathbb{R}^+$ that satisfy that for every $s\in \mathbb{R}^+$ and every pair $(y,\Tilde{y})\in \mathbb{R}^d\times \mathbb{R}^d$
    $$|f(s,y)-f(s,\Tilde{y})|\leq \mathfrak{v}(s)|y-\Tilde{y}|\text{ ; }|f(s,y)|\leq \mathfrak{u}(s).$$ 

Furthermore, allow us to consider the following hypothesi:

\begin{itemize}
    \item [\textbf{(c1)}]  (\ref{89}) admits a non uniform dichotomy, \textit{i.e.} there are two invariant complementary projectors $P(\cdot)$  and $Q(\cdot)$ such that $P(t)+Q(t)=I$ for every $t\geq 0$, a continuous function $K:[0,+\infty[\to [0,\infty[$ and a decreasing $C^1$ function $h:[0,\infty[\to ]0,1]$, such that $h(0)=1$ and $\lim_{t\to \infty}h(t)=0$ and they satisfy
    $$\left\{ \begin{array}{lc}
            \norm{X(t,s)P(s)}\leq K(s)
            \left( \mathlarger{\frac{h(t)}{h(s)}}\right), &\forall t \geq s\geq 0 \\
            \\ \norm{X(t,s)Q(s)}\leq K(s)
            \left( \mathlarger{\frac{h(s)}{h(t)}}\right), &\forall 0\leq t\leq s.
             \end{array}
   \right.$$
    \item [\textbf{(c2)}] 
    $$\int_{0}^{t}\norm{X(t,s)P(s)}\mathfrak{u}(s)ds+\int_t^\infty \norm{X(t,s)Q(s)}\mathfrak{u}(s)ds\leq \mathfrak{p}<\infty\text{ , for every }t \in \mathbb{R}^+.$$
    
    \item [\textbf{(c3)}] 
    $$\int_{0}^{t}\norm{X(t,s)P(s)}\mathfrak{v}(s)ds+\int_t^\infty \norm{X(t,s)Q(s)}\mathfrak{v}(s)ds\leq \mathfrak{q}<1\text{ , for every }t \in \mathbb{R}^+.$$
    
    \item [\textbf{(c4)}] The map $u\mapsto f(t,u)$ and its derivatives respect to $u$ up to the order $r$ ($r\geq 1$) are continuous functions of $(t,u)\in \mathbb{R}^+\times \mathbb{R}^d$ and $\sup_{u\in \mathbb{R}^d}\norm{\frac{\partial f}{\partial u}(t,u)}<+\infty $ is bounded.
    
    \item [\textbf{(c5)}] For every fixed $\tau\in \mathbb{R}^+$, the functions  $K,h$ and $\mathfrak{v}$ satisfy
    $$\int_\tau^\infty K(s)h(s)\mathfrak{v}(s)\exp\left(\int_\tau^s \norm{A(r)}+\mathfrak{v}(r)dr\right)<+\infty.$$
\end{itemize}

\begin{remark}
The projectors $P$ and $Q$ have been called invariant for (\ref{89}), which means that for every $t,s\in \mathbb{R}^+$ they verify:
$$P(t)X(t,s)=X(t,s)P(s)\text{ and }Q(t)X(t,s)=X(t,s)Q(s).$$
\end{remark}

\begin{definition}
Green's operator associated to (\ref{89}) and the dichotomy \textbf{(c1)} is the matrix function $\mathcal{G}:\mathbb{R}^+\times\mathbb{R}^+\to \mathcal{M}_{d}(\mathbb{R})$ given by:
\begin{equation*}
    \mathcal{G}(t,s)= \left\{ \begin{array}{cc}
             X(t,s)P(s)  & \forall t \geq s\geq 0, \\
             \\ -X(t,s)Q(s) & \forall 0\leq t<s,
             \end{array}
   \right.
\end{equation*}
and it is easily deduced that
$$\frac{\partial\mathcal{G}}{\partial t}(t,s)=A(t)\mathcal{G}(t,s).$$
\end{definition}

\begin{remark}
With this notation, conditions \textbf{(c2)} and \textbf{(c3)} are rewritten as:
\begin{itemize}
    \item [\textbf{(c2)}] 
    $$\int_{0}^{\infty}\norm{\mathcal{G}(t,s)}\mathfrak{u}(s)ds\leq \mathfrak{p}<\infty\text{ , for every }t \in \mathbb{R}^+.$$
    
    \item [\textbf{(c3)}]
    $$\int_{0}^{\infty}\norm{\mathcal{G}(t,s)}\mathfrak{v}(s)ds\leq \mathfrak{q}<1\text{ , for every }t \in \mathbb{R}^+.$$
\end{itemize}
\end{remark}

\section{Topological equivalence}

In this section we prove that the linear system \textnormal{(\ref{89})} and \textnormal{(\ref{90})} are topologically equivalents on the positive half line. In order to state this result, we first recall the following definition.

\begin{definition}\label{208}
Let $J\subset \mathbb{R}$ an interval. Systems (\ref{89}) and (\ref{90}) are $J$-topologically equivalent if there is a function $H:J\times \mathbb{R}^d\to \mathbb{R}^d$ that satisfies
\begin{itemize}
    \item [i)] If $x(t)$ is solution of (\ref{89}), then $H[t,x(t)]$ is solution of (\ref{90}).
    \item [ii)] $H(t,u)-u$ is bounded on $J\times \mathbb{R}^d$.
    \item [iii)] For every fixed $\tau \in J$, the map $u\mapsto H(\tau,u)$ is an homeomophism of $\mathbb{R}^d$.
\end{itemize}
Moreover, the function $u\mapsto G(\tau,u)=H^{-1}(\tau,u)$ verifies conditions ii) and iii) and maps solutions of (\ref{90}) on solutions of (\ref{89}).
\end{definition}

\begin{theorem}\label{167}
If conditions \textbf{(c1)}, \textbf{(c2)} and \textbf{(c3)} hold, then (\ref{89}) and (\ref{90}) are topologically equivalent on  $\mathbb{R}^+$.
\end{theorem}

\begin{proof} We develop the proof in several steps.\\

\textit{Step 1: Auxiliary functions.} We define $w^*:\mathbb{R}^+\to \mathbb{R}^d$ for $(\tau,\eta)\in \mathbb{R}^+\times\mathbb{R}^d$ by
\begin{eqnarray}\label{166}
    w^*(t;(\tau,\eta))&=&-\int_0^\infty \mathcal{G}(t,s)f(s,y(s,\tau,\eta))ds \nonumber \\
    \\
    &=&-\int_0^t X(t,s)P(s)f(s,y(s,\tau,\eta))ds+\int_t^\infty X(t,s)Q(s)f(s,y(s,\tau,\eta))ds\nonumber.
\end{eqnarray}

We also define $T:BC(\mathbb{R}^+,\mathbb{R}^d)\to BC(\mathbb{R}^+,\mathbb{R}^d)$ for $(\tau,\xi)\in \mathbb{R}^+\times \mathbb{R}^d$ by
\begin{equation*}
    T(\phi)(t;(\tau,\xi))=\int_0^\infty \mathcal{G}(t,s)f(s,x(s,\tau,\xi)+\phi(s))ds.
\end{equation*}

$T$ is well defined by \textbf{(c2)}. Using condition \textbf{(c3)} we obtain
\begin{eqnarray*}
    \left| T(\phi)(t;(\tau,\xi))- T(\psi)(t;(\tau,\xi))\right|&\leq&\int_{0}^{\infty}\norm{\mathcal{G}(t,s)}\left|\phi(s)-\psi(s)\right|\mathfrak{v}(s)ds\\
    \\
    &\leq&\mathfrak{q}\norm{\phi-\psi}_\infty,
\end{eqnarray*}
hence, using Banach's fixed point Theorem, we obtain the existence of an unique fixed point
\begin{equation*}
    z^*(t;(\tau,\xi))=\int_0^\infty \mathcal{G}(t,s)f(s,x(s,\tau,\xi)+z^*(s;(\tau,\xi)))ds.
\end{equation*}

It is easy to verify $t\mapsto w^*(t;(\tau,\eta))$ is solution to the initial value problem
$$\left\{ \begin{array}{ccl}
            \dot{w}(t) & = & A(t)w(t)-f(t,y(t,\tau,\eta))\\
            \\w(0) & = & - \mathlarger{\int}_{0}^\infty X(0,s)Q(s)f(s,y(s,\tau,\eta))ds,
\end{array}
\right.$$
while $t\mapsto z^*(t;(\tau,\xi))$ is respectively solution to the initial value problem:
$$\left\{ \begin{array}{ccl}
            \dot{z}(t) & = & A(t)z(t)-f(t,x(t,\tau,\xi)+z(t))\\
            \\z(0) & = & - \mathlarger{\int}_{0}^\infty X(0,s)Q(s)f(s,x(s,\tau,\xi)+z^*(s;(\tau,\xi)))ds,
\end{array}
\right.$$
furthermore, by using \textbf{(c2)}, the maps $t\mapsto w^*(t;(t,\eta))$ and $t\mapsto z^*(t;(t,\xi))$ are uniformly bounded. \\

\textit{Step 2: Construct maps $H$ and $G$.} By uniqueness of solutions
\begin{equation}\label{165}
    x(t,\tau,\xi)=x(t,s,x(s,\tau,\xi))\text{  , for every }t,s,\tau\in \mathbb{R}^+,
\end{equation}
and
\begin{equation}\label{164}
    z^*(t;(\tau,\xi))=z^*(t;(s,x(s,\tau,\xi)))\text{  , for every }t,s,\tau\in \mathbb{R}^+.
\end{equation}

For every fixed $t\in \mathbb{R}^+$ we define $H(t,\cdot):\mathbb{R}^d\to \mathbb{R}^d$ and $G(t,\cdot):\mathbb{R}^d\to \mathbb{R}^d$ by
\begin{equation}\label{160}\left\{ \begin{array}{ccl}
            H(t,\xi) & = & \xi + \mathlarger{\int}_{0}^{\infty} \mathcal{G}(t,s)f(s,x(s,t,\xi)+z^*(s;(t,\xi)))ds \\
            \\ & = & \xi+z^*(t;(t,\xi))
\end{array}
\right.\end{equation}
and
\begin{equation}\label{161}\left\{ \begin{array}{ccl}
            G(t,\eta) & = & \eta - \mathlarger{\int}_{0}^{\infty} \mathcal{G}(t,s)f(s,y(s,t,\eta))ds \\
            \\ & = & \eta+w^*(t;(t,\eta)).
\end{array}
\right.\end{equation}

Using \textbf{(c2)}, it follows immediately that $H(t,\xi)-\xi$ and $ G(t,\eta)-\eta$ are bounded on $\mathbb{R}^+\times \mathbb{R}^d$, hence $H$ and $G$ satisfy condition ii) from Definition \ref{208}. In order to study additional properties of the map $G$, note that for $\tau \geq t$
$$y(t,\tau,\eta)=X(t,\tau)\eta-\int_t^\tau X(t,s)f(s,y(s,\tau,\eta))ds, $$
or equivalently
\begin{eqnarray*}
    X(\tau,t)y(t,\tau,\eta)&=&\eta-\int_t^\tau X(\tau,s)f(s,y(s,\tau,\eta))ds\\
    \\
    &=&\eta-\int_t^\tau X(\tau,s)P(s)f(s,y(s,\tau,\eta))ds-\int_t^\tau X(\tau,s)Q(s)f(s,y(s,\tau,\eta))ds,
\end{eqnarray*}
in particular, for $t=0$ we obtain
\begin{eqnarray*}
    X(\tau,0)y(0,\tau,\eta)&=&\eta-\int_0^\tau X(\tau,s)P(s)f(s,y(s,\tau,\eta))ds-\int_0^\tau X(\tau,s)Q(s)f(s,y(s,\tau,\eta))ds\\
    \\
    &=&\eta-\int_0^\tau X(\tau,s)P(s)f(s,y(s,\tau,\eta))ds\\
    \\
    & &+\int_\tau^\infty X(\tau,s)Q(s)f(s,y(s,\tau,\eta))ds-\int_0^\infty X(\tau,s)Q(s)f(s,y(s,\tau,\eta))ds\\
    \\
     &=&\eta-\int_0^\infty \mathcal{G}(\tau,s)f(s,y(s,\tau,\eta))ds\\
     \\&&-X(\tau,0)\int_0^\infty X(0,s)Q(s)f(s,y(s,\tau,\eta))ds\\
     \\
     &=&G(\tau,\eta)-X(\tau,0)w^*(0;(\tau,\eta)),
\end{eqnarray*}
thus
\begin{equation}\label{169}
    G(\tau,\eta)=X(\tau,0)\left\{y(0,\tau,\eta)+w^*(0;(\tau,\eta)) \right\}.
\end{equation}

\textit{Step 3: $H$ maps solutions of (\ref{89}) on solutions of (\ref{90}) and $G$ maps solutions of (\ref{90}) on solutions of (\ref{89})}. By uniqueness of solutions, by simple differentiation we obtain
\begin{equation}\label{162}
    H[t,x(t,\tau,\xi)]=y(t,\tau,H(\tau,\xi))
\end{equation}
and
\begin{equation}\label{163}
    G[t,y(t,\tau,\eta)]=x(t,\tau,G(\tau,\eta))=X(t,\tau)G(\tau,\eta),
\end{equation}
hence both maps satisfy condition i) of Definition \ref{208} respectively.\\

\textit{Step 4: $u\mapsto G(t,u)$ and $u\mapsto H(t,u)$ are bijective for every fixed $t\geq 0$.}\\ 

First we show $H(t,G(t,\eta))=\eta$ for every $t\geq 0$. Using (\ref{160}) and (\ref{161})
\begin{eqnarray*}
H(t,G[t,y(t,\tau,\eta)])&=&G[t,y(t,\tau,\eta)]\\
\\
&&+\int_0^\infty \mathcal{G}(t,s)f(s,x(s,t,G[t,y(t,\tau,\eta)])+z^*(s;(t,G[t,y(t,\tau,\eta)])))\\
\\
&=&y(t,\tau,\eta)-\int_0^\infty \mathcal{G}(t,s)f(s,y(s,\tau,\eta))ds\\
\\
& &+\int_0^\infty \mathcal{G}(t,s)f(s,x(s,t,G[t,y(t,\tau,\eta)])+z^*(s;(t,G[t,y(t,\tau,\eta)]))).
\end{eqnarray*}

Define $\omega(t)=\left|H[t,G[t,y(t,\tau,\eta)]]-y(t,\tau,\eta)\right|$. Note that
\begin{eqnarray*}
\omega(t)\leq \left|H[t,G[t,y(t,\tau,\eta)]]-G[t,y(t,\tau,\eta)]\right|+\left|G[t,y(t,\tau,\eta)]-y(t,\tau,\eta)\right|<\infty,
\end{eqnarray*}
since both $H$ and $G$ satisfy condition ii) of Definition \ref{208}. Thus, using \textbf{(c3)}, along with the previous expression and the identities (\ref{164}), (\ref{160}) and (\ref{163}), for an arbitrary $t\in\mathbb{R}^+$ we have
\begin{eqnarray*}
\omega(t)&=&\left|\int_0^\infty \mathcal{G}(t,s)\left\{f(s,x(s,t,G[t,y(t,\tau,\eta)])+z^*(s;(t,G[t,y(t,\tau,\eta)])))- f(s,y(s,\tau,\eta))\right\}ds \right|\\
\\
&\leq&\int_0^\infty \norm{\mathcal{G}(t,s)}\mathfrak{v}(s)\left|x(s,t,G[t,y(t,\tau,\eta)])+z^*(s;(t,G[t,y(t,\tau,\eta)]))- y(s,\tau,\eta)\right| ds\\
\\
&\leq&\int_0^\infty \norm{\mathcal{G}(t,s)}\mathfrak{v}(s)\left|x(s,t,x(t,\tau,G(\tau,\eta)))+z^*(s;(t,x(t,\tau,G(\tau,\eta))))-y(s,\tau,\eta)\right| ds\\
\\
&\leq&\int_0^\infty \norm{\mathcal{G}(t,s)}\mathfrak{v}(s)\left|x(s,\tau,G(\tau,\eta))+z^*(s;(s,x(s,\tau,G(\tau,\eta))))- y(s,\tau,\eta)\right| ds\\
\\
&=&\int_0^\infty \norm{\mathcal{G}(t,s)}\mathfrak{v}(s)\left|H[s,G[s,y(s,\tau,\eta)])]- y(s,\tau,\eta)\right| ds\\
\\
&=&\int_0^\infty \norm{\mathcal{G}(t,s)}\mathfrak{v}(s)\omega(s)ds\leq \mathfrak{q}\cdot \sup_{s\in \mathbb{R}^+}\{ \omega(s)\}.
\end{eqnarray*}

Thus $\omega(t)=0$ for every $t\in \mathbb{R}^+$, otherwise we get a contradiction. In particular, taking  $t=\tau$ we obtain $H(\tau,G(\tau,\eta))=\eta$. Now we show $G(t,H(t,\xi))=\xi$. Indeed, using (\ref{165}), (\ref{164}), (\ref{160}), (\ref{161}) and (\ref{162}) we get
\begin{eqnarray*}
G[t,H[t,x(t,\tau,\xi)]]&=&H[t,x(t,\tau,\xi)]-\int_0^\infty \mathcal{G}(t,s)f(s,y(s,t,H[t,x(t,\tau,\xi)]))ds\\
\\
&=&H[t,x(t,\tau,\xi)]-\int_0^\infty \mathcal{G}(t,s)f(s,y(s,t,y(t,\tau,H(\tau,\xi))))ds\\
\\
&=&H[t,x(t,\tau,\xi)]-\int_0^\infty \mathcal{G}(t,s)f(s,y(s,\tau,H(\tau,\xi)))ds\\
\\
&=&x(t,\tau,\xi)+\int_0^\infty \mathcal{G}(t,s)f(s,x(s,t,x(t,\tau,\xi))+z^*(s;(t,x(t,\tau,\xi))))ds\\
\\
& &-\int_0^\infty \mathcal{G}(t,s)f(s,y(s,\tau,H(\tau,\xi)))ds\\
\\
&=&x(t,\tau,\xi)+\int_0^\infty \mathcal{G}(t,s)f(s,x(s,\tau,x(\tau,\tau,\xi))+z^*(s;(s,x(s,\tau,\xi))))ds\\
\\
& &-\int_0^\infty \mathcal{G}(t,s)f(s,y(s,\tau,H(\tau,\xi)))ds\\
\\
&=&x(t,\tau,\xi)\\
\\&&+\int_0^\infty \mathcal{G}(t,s)\left\{f(s,x(s,\tau,\xi)+z^*(s;(s,x(s,\tau,\xi))))-f(s,y(s,\tau,H(\tau,\xi)))\right\}ds\\
\\
&=&x(t,\tau,\xi)+\int_0^\infty \mathcal{G}(t,s)\left\{f(s,H(s,x(s,\tau,\xi)))-f(s,y(s,\tau,H(\tau,\xi)))\right\}ds\\
\\
&=&x(t,\tau,\xi)+\int_0^\infty \mathcal{G}(t,s)\left\{f(s,y(s,\tau,H(\tau,\xi)))-f(s,y(s,\tau,H(\tau,\xi)))\right\}ds\\
\\
&=&x(t,\tau,\xi),
\end{eqnarray*}
evaluating on $t=\tau$ we obtain $G(\tau,H(\tau,\xi))=\xi$.\\

\textit{Step 5: $u\mapsto G(t,u)$ is continuous.} \\

It is enough to show $\eta\mapsto w^*(t;(t,\eta))$ is continuous for every $t\geq 0$, since $G(t,\eta)=\eta+w^*(t;(t,\eta))$. Let $\eta\in \mathbb{R}^d$ and $\{ \eta_n\}_{n\in \mathbb{N}}\subset\mathbb{R}^d$ be a sequence such that $\lim_{n\to \infty}\eta_n=\eta$. Fix $t,\tau\in \mathbb{R}^+$ and define $(a_n)_{n\in \mathbb{N}}$ the sequence of functions over $\mathbb{R}^+$ given by
$$a_n(s)=\mathcal{G}(t,s)f(s,y(s,\tau,\eta_n)),$$
notice that
$$|a_n(s)|\leq \norm{\mathcal{G}(t,s)} \mathfrak{u}(s)\text{ , for every }s\in \mathbb{R}^+ \text{ and }n\in\mathbb{N}.$$

On the other hand, as $u\mapsto f(s,u)$ and $\xi\mapsto y(s,\tau,\xi)$ are continuous, it is clear $(a_n)_{n\in \mathbb{N}}$ converges pointwise to $a:\mathbb{R}^+\to \mathbb{R}^d$ given by
$$a(s)=\mathcal{G}(t,s)f(s,y(s,\tau,\eta)).$$

Thus, by Lebesgue's dominated convergence Theorem we have
\begin{eqnarray*}
\lim_{n\to \infty}w^*(t;(\tau,\eta_n))&=&\lim_{n\to \infty}-\int_0^\infty \mathcal{G}(t,s)f(s,y(s,\tau,\eta_n))ds\\
\\&=&-\lim_{n\to \infty}\int_0^\infty a_n(s)ds=-\int_0^\infty a(s)ds=w^*(t;(\tau,\eta)).
\end{eqnarray*}

Hence $\eta\mapsto w^*(t;(\tau,\eta))$ is continuous and in particular $\eta\mapsto w^*(t;(t,\eta))$ is also continuous.\\

\textit{Step 6: $u\mapsto H(t,u)$ is continuous.}\\

It is enough to show $\xi\mapsto z^*(t;(t,\xi))$ is continuous for every $t\geq 0$.\\

Let $\xi\in \mathbb{R}^d$ and a sequence $\{ \xi_n\}_{n\in \mathbb{N}}\subset\mathbb{R}^d$ such that $\lim_{n\to \infty}\xi_n=\xi$. Let $u\mapsto \phi(t;(\tau,u))$ be a continuous function for $t,\tau\in \mathbb{R}^+$ fixed. We define
$$b_n(s)=\mathcal{G}(t,s)f(s,x(s,\tau,\xi_n)+\phi(s;(\tau,\xi_n))),$$
note it satisfies
$$|b_n(s)|\leq \norm{\mathcal{G}(t,s)} \mathfrak{u}(s)\text{ , para todo }s\in \mathbb{R}^+\text{ y }n\in\mathbb{N},$$
and
$$\lim_{n\to \infty}b_n(s)=\mathcal{G}(t,s)f(s,x(s,\tau,\xi)+\phi(s;(\tau,\xi))):=b(s).$$

Using Lebesgue's dominated convergence Theorem we have
\begin{eqnarray*}
    \lim_{n\to \infty}(T \phi)(t;(\tau,\xi_n))&=&\lim_{n\to\infty} \int_0^\infty \mathcal{G}(t,s)f(s,x(s,\tau,\xi_n)+\phi(s;(\tau,\xi_n)))ds\\
    \\
    &=&\lim_{n\to \infty}\int_0^\infty b_n(s)ds
    =\int_0^\infty b(s)ds=(T \phi)(t;(\tau,\xi)),
\end{eqnarray*}
thus $\xi\mapsto (T \phi)(t;(\tau,\xi))$ is continuous, and hence it's fixed point $\xi\mapsto z^*(t;(\tau,\xi))$ is continuous as well, which allows us to conclude $H$ is continuous and so it is an homeomorphism. In conclusion (\ref{89}) and (\ref{90}) are topologically equivalent on $\mathbb{R}^+$.
\end{proof}

\section{Differentiability of topological equivalence under a dichotomy}

In this section we prove that the topological equivalence is of class $C^1.$ Our approach does not impose resonance conditions or spectral gaps. 

We recall of the definition, introduced on \cite{Castaneda3}, of $C^r-$ topologically equivalent on the positive half line .

\begin{definition}
The systems (\ref{89}) and (\ref{90}) are $C^r$-topologically equivalent on $\mathbb{R}^+$ if they are topologically equivalent on $\mathbb{R}^+$ with the map $u\mapsto H(t,u)$, which is a diffeomorphism of class $C^r$, with $r\geq 1$, for every fixed $t\geq 0$.
\end{definition}

\begin{lemma}\label{168}
If conditions \textbf{(c1)-(c5)} hold, with $r=1$ on \textbf{(c4)}, then $\eta\mapsto w^*(0;(\tau,\eta))$ defined on (\ref{166}) is a $C^1$ map.
\end{lemma}

\begin{proof} Fix $\eta\in \mathbb{R}^d$ and let $(\delta_n)_{n\in \mathbb{N}}\subset \mathbb{R}^d$ be a properly convergent to zero sequence. Fix $\tau\in \mathbb{R}^+$ and define
$$\varphi_n(s)=\mathcal{G}(0,s)\frac{f(s,y(s,\tau,\eta+\delta_n))-f(s,y(s,\tau,\eta))-\frac{\partial f}{\partial u}(s,y(s,\tau,\eta))\frac{\partial y}{\partial \eta}(s,\tau,\eta)\delta_n}{|\delta_n|}.$$

As $\eta\mapsto y(s,\tau,\eta)$ is continuous, then $\lim_{n\to \infty}y(s,\tau,\eta+\delta_n)=y(s,\tau,\eta)$. Using \textbf{(c4)} it follows from classic results of differentiability respect to the initial conditions (see for example Theorem 4.1 in \cite{Hartman}) that $\eta\mapsto y(s,\tau,\eta)$ is differentiable. Again by \textbf{(c4)} we obtain
$$\lim_{n\to\infty}\varphi_n(s)=0.$$

Also note
\begin{eqnarray*}
    \norm{\frac{\partial f}{\partial u}(s,u)}=\lim_{\delta\to 0}\frac{|f(s,u+\delta)-f(s,u)|}{|\delta|}\leq\lim_{\delta\to 0}\mathfrak{v}(s)=\mathfrak{v}(s),
\end{eqnarray*}
hence
\begin{eqnarray*}
    |\varphi_n(s)|&\leq& \norm{\mathcal{G}(0,s)}\frac{\left|f(s,y(s,\tau,\eta+\delta_n))-f(s,y(s,\tau,\eta))\right|+\left|\frac{\partial f}{\partial u}(s y(s,\tau,\eta))\frac{\partial y}{\partial \eta}(s,\tau,\eta)\delta_n\right|}{|\delta_n|}\\
    \\
    &\leq& \norm{\mathcal{G}(0,s)}\frac{\mathfrak{v}(s)\left|y(s,\tau,\eta+\delta_n)-y(s,\tau,\eta)\right|+\mathfrak{v}(s)\left|\frac{\partial y}{\partial \eta}(s,\tau,\eta)\delta_n\right|}{|\delta_n|}\\
    \\
    &\leq& \norm{\mathcal{G}(0,s)}\mathfrak{v}(s)\left(\frac{\left|y(s,\tau,\eta+\delta_n)-y(s,\tau,\eta)\right|}{|\delta_n|}+\norm{\frac{\partial y}{\partial \eta}(s,\tau,\eta)}\right).\\
\end{eqnarray*}

Let $\Tilde{\eta}\in \mathbb{R}^d$ and $s\geq \tau$. We know
$$y(s,\tau,\eta)=y(\tau,\tau,\eta)+\int_\tau^s \dot{y}(r,\tau,\eta)dr=\eta+\int_\tau^s \dot{y}(r,\tau,\eta)dr,$$
hence
$$y(s,\tau,\eta)-y(s,\tau,\Tilde{\eta})=\eta-\Tilde{\eta}+\int_\tau^s \dot{y}(r,\tau,\eta)- \dot{y}(r,\tau,\Tilde{\eta})dr.$$

Defining $\mathfrak{z}(s)=\dot{y}(s,\tau,\eta)-\dot{y}(s,\tau,\Tilde{\eta})$, the previous expression implies
$$|y(s,\tau,\eta)-y(s,\tau,\Tilde{\eta})|\leq|\eta-\Tilde{\eta}|+\int_\tau^s |\mathfrak{z}(r)|dr.$$

Note
$$\dot{y}(s,\tau,\eta)-\dot{y}(s,\tau,\Tilde{\eta})=A(s)\left[y(s,\tau,\eta)-y(s,\tau,\Tilde{\eta}) \right]+f(s,y(s,\tau,\eta))-f(s,y(s,\tau,\Tilde{\eta})),$$
from where it follows that
\begin{eqnarray*}
|\mathfrak{z}(s)|&\leq& \norm{A(s)}|y(s,\tau,\eta)-y(s,\tau,\Tilde{\eta})|+|f(s,y(s,\tau,\eta))-f(s,y(s,\tau,\Tilde{\eta}))|\\
\\&\leq& \norm{A(s)}|y(s,\tau,\eta)-y(s,\tau,\Tilde{\eta})|+\mathfrak{v}(s)|y(s,\tau,\eta)-y(s,\tau,\Tilde{\eta})|\\
\\&\leq& \left(\norm{A(s)}+\mathfrak{v}(s)\right)|y(s,\tau,\eta)-y(s,\tau,\Tilde{\eta})|\\
\\&\leq& \left(\norm{A(s)}+\mathfrak{v}(s)\right)\left[|\eta-\Tilde{\eta}|+\int_\tau^s |\mathfrak{z}(r)|dr\right].
\end{eqnarray*}

Define $\mathcal{Z}(s)=|\eta-\Tilde{\eta}|+\int_\tau^s|\mathfrak{z}(r)|dr$. Note  $s> \tau \Rightarrow \mathcal{Z}\neq 0$, thus, the previous expression becomes
$$\frac{\mathcal{Z}'(s)}{\mathcal{Z}(s)}\leq \left(\norm{A(s)}+\mathfrak{v}(s)\right),$$
from where
$$\log\left(\mathcal{Z}(s)\right)-\log \left(\mathcal{Z}(\tau)\right)=\log \left(\frac{\mathcal{Z}(s)}{|\eta-\Tilde{\eta}|}\right)\leq \int_\tau^s \norm{A(r)}+\mathfrak{v}(r)dr,$$
hence
$$|y(s,\tau,\eta)-y(s,\tau,\Tilde{\eta})|\leq|\eta-\Tilde{\eta}|+\int_\tau^s |z(r)|dr=\mathcal{Z}(s)\leq |\eta-\Tilde{\eta}|\exp\left(\int_\tau^s \norm{A(r)}+\mathfrak{v}(r)dr\right),$$
so, for $s\geq \tau$
$$\frac{|y(s,\tau,\eta)-y(s,\tau,\eta+\delta_n)|}{|\delta_n|}\leq \exp\left(\int_\tau^s \norm{A(r)}+\mathfrak{v}(r)dr\right),$$
which in particular implies
$$\norm{\frac{\partial y}{\partial \eta}(s,\tau,\eta)}\leq \exp\left(\int_\tau^s \norm{A(r)}+\mathfrak{v}(r)dr\right).$$

In conclusion, for $s\geq \tau$ we have
\begin{eqnarray*}
    |\varphi_n(s)|    &\leq& \norm{\mathcal{G}(0,s)}\mathfrak{v}(s)\left(\frac{\left|y(s,\tau,\eta+\delta_n)-y(s,\tau,\eta)\right|}{|\delta_n|}+\norm{\frac{\partial y}{\partial \eta}(s,\tau,\eta)}\right)\\
    \\&\leq&2 \norm{\mathcal{G}(0,s)}\mathfrak{v}(s)\exp\left(\int_\tau^s \norm{A(r)}+\mathfrak{v}(r)dr\right)\\
    \\&\leq &2K(s)h(s)\mathfrak{v}(s)\exp\left(\int_\tau^s \norm{A(r)}+\mathfrak{v}(r)dr\right).
\end{eqnarray*}

On the other hand, on the compact $[0,\tau]$ the sequence of continuous functions $s \mapsto \frac{\left|y(s,\tau,\eta+\delta_n)-y(s,\tau,\eta)\right|}{|\delta_n|}$ converge pointwise to $s\mapsto \norm{\frac{\partial y}{\partial \eta}(s,\tau,\eta)}$ when $n\to \infty$, which is also continuous, hence the convergence is uniform, \textit{i.e. } there exist $\hat{n}\in \mathbb{N}$ such that $n\geq \hat{n}$ and $s\in [0,\tau]$ we have
$$\frac{\left|y(s,\tau,\eta+\delta_n)-y(s,\tau,\eta)\right|}{|\delta_n|}+\norm{\frac{\partial y}{\partial \eta}(s,\tau,\eta)}\leq 2 \norm{\frac{\partial y}{\partial \eta}(s,\tau,\eta)} +1.$$

Hence, for $n\geq \hat{n}$ and $s\in [0,\tau]$ we have
\begin{eqnarray*}
    |\varphi_n(s)|    &\leq& \norm{\mathcal{G}(0,s)}\mathfrak{v}(s)\left(\frac{\left|y(s,\tau,\eta+\delta_n)-y(s,\tau,\eta)\right|}{|\delta_n|}+\norm{\frac{\partial y}{\partial \eta}(s,\tau,\eta)}\right)\\
    \\&\leq &K(s)h(s)\mathfrak{v}(s)\left( 2 \norm{\frac{\partial y}{\partial \eta}(s,\tau,\eta)} +1\right).
\end{eqnarray*}

Summarizing, if we define $\mathcal{F}:\mathbb{R}^+\to \mathbb{R}$ by
\begin{equation*}
    \mathcal{F}(s)= \left\{ \begin{array}{cc}
             K(s)h(s)\mathfrak{v}(s)\left( 2 \norm{\frac{\partial y}{\partial \eta}(s,\tau,\eta)} +1\right)  & \tau\geq s \geq 0, \\
             \\ 2K(s)h(s)\mathfrak{v}(s)\exp\left(\int_\tau^s \norm{A(r)}+\mathfrak{v}(r)dr\right) & \forall s\geq \tau,
             \end{array}
   \right.
\end{equation*}
we get $|\varphi_n(s)|\leq \mathcal{F}(s)$ for every $n\geq\hat{n}$. Note that using \textbf{(c5)}
\begin{eqnarray*}
\int_0^\infty \mathcal{F}(s)ds&\leq&\int_0^\tau  K(s)h(s)\mathfrak{v}(s)\left( 2 \norm{\frac{\partial y}{\partial \eta}(s,\tau,\eta)} +1\right) ds\\
\\& &+2\int_\tau^\infty K(s)h(s)\mathfrak{v}(s)\exp\left(\int_\tau^s \norm{A(r)}+\mathfrak{v}(r)dr\right) ds <+\infty.
\end{eqnarray*}

Thus, using Lebesgue's dominated convergence Theorem we obtain

$$\lim_{n\to \infty}\frac{ w^*(0;(\tau,\eta+\delta_n))-w^*(0;(\tau,\eta))+\left[ \int_0^\infty \mathcal{G}(0,s)\frac{\partial f}{\partial u}(s,y(s,\tau,\eta))\frac{\partial y}{\partial \eta}(s,\tau,\eta)ds\right]\delta_n}{|\delta_n|}$$
    
$$=\lim_{n\to \infty}\int_0^\infty-\varphi_n(s)ds=-\int_0^\infty \left(\lim_{n\to \infty}\varphi_n(s)\right)ds=0,$$
hence $\eta\mapsto w^*(0;(\tau,\eta))$ is differentiable.
\end{proof}

\begin{corollary}\label{220}
If conditions \textbf{(c1)-(c5)} hold, with $r=1$ on \textbf{(c4)}, then for every fixed $\tau\in \mathbb{R}^+$ we have
$$ \frac{\partial w^*(0;(\tau,\eta))}{\partial \eta}=-\int_{0}^\infty \mathcal{G}(0,s)\frac{\partial f}{\partial u}(s,y(s,\tau,\eta))\frac{\partial y}{\partial \eta}(s,\tau,\eta)ds.$$
\end{corollary}

\begin{theorem}\label{225}
If conditions \textbf{(c1)-(c5)} hold, with $r=1$ on \textbf{(c4)}, then the systems (\ref{89}) and (\ref{90}) are $C^1$-topologically equivalent on $\mathbb{R}^+$.
\end{theorem}

\begin{proof} By Theorem \ref{167} we know they are topologically equivalent. By Lemma \ref{168} we know $\eta\mapsto w^*(0;(\tau,\eta))$ is a $C^1$ map and by classic results of differentiability respect to the initial conditions (see for example Theorem 4.1 in \cite{Hartman}) we know $\eta \mapsto y(0,\tau,\eta)$ is as well a $C^1$ map, which along with the expression (\ref{169}) allows us to conclude $\eta\mapsto G(\tau,\eta)$ is a $C^1$ map.\\


Furthermore, as $G$ is a topological equivalence, then $\xi \mapsto G(\tau,\xi)-\xi$ is bounded, thus $G(\tau,\xi)\to \infty$ when $|\xi|\to \infty$. This fact, along with the previous discussion implies by \cite[Corollary 2.1]{Plastock} that $\xi \mapsto G(\tau,\xi)$ is a $C^1$ diffeomorphism. Moreover, as $G(\tau,H(\tau,\xi))=\xi$, we have
$$\frac{\partial G}{\partial \xi}(\tau,H(\tau,\xi))\frac{\partial H}{\partial \xi}(\tau,\xi)=I,$$
and so $\frac{\partial H}{\partial \xi}(\tau,\xi)=\left[ \frac{\partial G}{\partial \xi}(\tau,H(\tau,\xi))\right]^{-1}$, which completes the proof.
\end{proof}

\begin{corollary}
If \textbf{(c4)} is satisfied with $r=1$. Furthermore, suppose (\ref{89}) admits an exponential dichotomy, \textit{i.e.} $P$ and $Q$ are constant complementary projectors, $K(s)=K>0$ for every $s\in \mathbb{R}^+$ and $h(s)=e^{-\lambda s}$, with $\lambda>0$. Suppose  $\mathfrak{v}(s)=\mathfrak{v}>0$ and $\mathfrak{u}(s)=\mathfrak{u}>0$ for every $s \in \mathbb{R}^+$. Then, if $2K\mathfrak{v}<\lambda$ and $M+ \mathfrak{v}<\lambda$, the systems (\ref{89}) and (\ref{90}) are $C^1$-topologically equivalent on $\mathbb{R}^+$.
\end{corollary}

\begin{proof} For an arbitrary $t\in \mathbb{R}^+$ we have
\begin{eqnarray*}
    \int_{0}^\infty \norm{\mathcal{G}(t,s)}\mathfrak{u}ds&=&\int_{0}^{t} \norm{X(t,s)P}\mathfrak{u}ds+\int_{t}^\infty \norm{X(t,s)Q}\mathfrak{u}ds\\
    \\
    &\leq&\int_{0}^{t}e^{-\lambda(t-s)}K\mathfrak{u}ds+\int_{t}^\infty e^{-\lambda(s-t)}K\mathfrak{u}ds\\
    \\
    &\leq&K\mathfrak{u}\frac{1-e^{-\lambda  t}}{\lambda}+\frac{K\mathfrak{u}}{\lambda}\leq \frac{2K\mathfrak{u}}{\lambda},
\end{eqnarray*}
thus \textbf{(c2)} is satisfied. Analogously
$$\int_{0}^\infty \norm{\mathcal{G}(t,s)}\mathfrak{v}ds\leq \frac{2K \mathfrak{v}}{\lambda}<1,$$
thus \textbf{(c3)} is verified. Finally, for an arbitrary $\tau \in \mathbb{R}^+$ we have
\begin{eqnarray*}
    \mathlarger{\int_{\tau}^{\infty}}K\mathfrak{v}e^{-\lambda s}\exp\left(\int_\tau^s \norm{A(r)} +\mathfrak{v}dr \right)ds&\leq&\mathlarger{\int_{\tau}^{\infty}}K\mathfrak{v}e^{-\lambda s}\exp\left(\int_\tau^s  M +\mathfrak{v}dr \right)ds\\
    \\&=&\mathlarger{\int_{\tau}^{\infty}}K\mathfrak{v}e^{-\lambda s}e^{(M +\mathfrak{v})(s-\tau)}ds\\
    \\&=&e^{-(M+\mathfrak{v})\tau}K \mathfrak{v}\mathlarger{\int_{\tau}^{\infty}}e^{(M +\mathfrak{v}-\lambda)s}ds<+\infty,
\end{eqnarray*}
so \textbf{(c5)} is satisfied. Applying Theorem \ref{225} the systems are $C^1$-topologically equivalent on $\mathbb{R}^+$.
\end{proof}

\begin{corollary}\label{228}
Suppose (\ref{89}) admits a nonuniform exponential dichotomy, \textit{i.e.} there are two complementary invariant projectors $P(\cdot)$ and $Q(\cdot)$ and constants constanes $C,\lambda, \varepsilon_1>0$ such that
$$\left\{ \begin{array}{lc}
            \norm{X(t,s)P(s)}\leq C
            e^{-\lambda(t-s)+\varepsilon_1 s}, &\forall t \geq s\geq 0 \\
            \\ \norm{X(t,s)Q(s)}\leq C
            e^{\lambda(t-s)+\varepsilon_1 s}, &\forall 0\leq t\leq s.
             \end{array}
   \right.$$
   
Furthermore, suppose that for each $t\in \mathbb{R}^+$ $u\mapsto f(t,u)$ is a $C^1$ map such that $u\mapsto \frac{\partial f}{\partial u}(t,u)$ is a bounded map that satisfies
\begin{equation}\label{227}
    |f(s,u)|\leq \kappa e^{-\varepsilon_0 s}
\end{equation}
and
\begin{equation}\label{226}
    \norm{\frac{\partial f}{\partial u}(s,u)}\leq \nu e^{-\varepsilon_1 s},
\end{equation}
for given $\nu ,\kappa >0$ y $\varepsilon_0 > \varepsilon_1 - \lambda$. Then, if $M<\lambda$, for $\nu>0$ being small enough the systems (\ref{89}) and (\ref{90}) are $C^1$-topologically equivalent on $\mathbb{R}^+.$
\end{corollary}

\begin{proof} Condition \textbf{(c1)} is easily verified with  $K(s)=Ce^{\varepsilon_1 s}$ and $h(s)=e^{-\lambda s}$. Note that condition (\ref{226}) implies
$$|f(s,y)-f(s,\Tilde{y})|\leq \nu e^{-\varepsilon_1 s}|y-\Tilde{y}|,$$
hence our general conditions are verified with  $\mathfrak{v}(s)=\nu e^{-\varepsilon_1 s}$ and $\mathfrak{u}(s)=\kappa e^{-\varepsilon_0 s}$. Thus condition \textbf{(c2)} is satisfied inmedietly and, for small enough $\nu$, \textbf{(c3)} is as well, meanwhile \textbf{(c4)} is granted by hypothesis. Denote
$$\Psi_\tau(s)=\exp \left( \int_\tau^s \norm{A(r)}+\mathfrak{v}(r) dr\right).$$

It is easy to see $s\geq \tau\geq t$ implies $\Psi_t(s)\leq \Psi_\tau(s)$, hence for a fixed $\tau\in \mathbb{R}^+$
$$\Psi_\tau(s)\leq \Psi_0(s)\leq e^{(M+\nu)s}.$$

Thus
\begin{eqnarray*}
    \int_{\tau}^\infty K(s)h(s)\mathfrak{v}(s)\Psi_\tau(s)ds&\leq& \int_{0}^\infty C \nu e^{-\lambda s}\Psi_0(s) ds\\
    \\
    &\leq&C\nu \int_{0}^\infty  e^{(M+\nu-\lambda)s}ds,
\end{eqnarray*}
which, as $M<\lambda$, is finite as soon as $\nu$ is small enough, particularly $0<\nu < \lambda-M$. Thus, condition \textbf{(c5)} is verified. Applying Theorem \ref{225} the corollary follows.
\end{proof}

\section{Second class of differentiability for the topological equivalence}

We once again consider the expression (\ref{169}) in order to study the second derivative of the homeomorphism of topological equivalence. Taking in account classic results of differentiability respect to the initial conditions (see for example Theorem 4.1 in \cite{Hartman}), we know the map $\eta\mapsto y(0,\tau,\eta)$ has the same class of differentiability as $u\mapsto f(\tau,u)$ for every  $\tau \in \mathbb{R}^+$ when conditions \textbf{(c1)-(c4)} are satisfied; hence the class of differentiability of the homeomorphism relies on the differentiability of the map $\eta \mapsto w^*(0;(\tau,\eta))$.

\begin{lemma}\label{224}
Suppose conditions \textbf{(c1)-(c5)} hold, with $r=2$ on \textbf{(c4)}. Furthermore, suppose there are functions $\mathfrak{V}:\mathbb{R}^+\to \mathbb{R}^+$ and $\pi_\tau:\mathbb{R}^+\to \mathbb{R}^+$ such that
\begin{equation}\label{221}
    \norm{\frac{\partial^2 f}{\partial u^2}(s,u)}\leq \mathfrak{V}(s)\text{ , for every }s\in \mathbb{R}^+
\end{equation}
and
\begin{equation}\label{222}
    \norm{\frac{\partial^2 y}{\partial \eta^2}(s,\tau,\eta)}\leq \pi_\tau(s)\text{ , for every }s\geq \tau.
\end{equation}

If for every fixed $\tau\in \mathbb{R}^+$ the previous functions verify
    \begin{equation}\label{223}
    \mathlarger{\int_{\tau}^{\infty}}\left(K(s)h(s)\left\{\pi_\tau(s)\mathfrak{v}(s)+\mathfrak{V}(s)\left[\exp\left(\int_\tau^s \norm{A(r)}+\mathfrak{v}(r)dr\right)\right]^2 \right\}\right)ds<+\infty,
    \end{equation}
then $\eta\mapsto w^*(0;(\tau,\eta))$ is a $C^2$ map.
\end{lemma}

\begin{proof} We follow the same strategy as in Lemma \ref{168}. In Corollary \ref{220} we established 
$$ \frac{\partial w^*(0;(\tau,\eta))}{\partial \eta}=-\int_{0}^\infty \mathcal{G}(0,s)\frac{\partial f}{\partial u}(s,y(s,\tau,\eta))\frac{\partial y}{\partial \eta}(s,\tau,\eta),$$
for every fixed $\tau\in \mathbb{R}^+$, which is granted by conditions \textbf{(c1)-(c5)}. Denote
$$\mathcal{O}_{\tau,\eta}(s)=\frac{\partial  f}{\partial u}(s,y(s,\tau,\eta))\frac{\partial^2 y}{\partial \eta^2}(s,\tau,\eta)+\frac{\partial ^2 f}{\partial u^2}(s,y(s,\tau,\eta))\left( \frac{\partial y}{\partial \eta}(s,\tau,\eta)\right)^2.$$

Fix $\eta\in \mathbb{R}^d$ and let $(\delta_n)_{n\in \mathbb{N}}\subset \mathbb{R}^d$ be a properly convergence to zero sequence. Choose $\tau\in \mathbb{R}^+$ and define the sequence of functions $(\mathfrak{F}_{n,\tau,\eta})_{n\in \mathbb{N}}$ over $\mathbb{R}^+$ given by
$$\mathfrak{F}_{n,\tau,\eta}(s)=\frac{\partial f}{\partial u}(s,y(s,\tau,\eta+\delta_n))\frac{\partial y}{\partial \eta}(s,\tau,\eta+\delta_n)-\frac{\partial f}{\partial u}(s,y(s,\tau,\eta))\frac{\partial y}{\partial \eta}(s,\tau,\eta).$$

We now define $(\psi_n)_{n\in \mathbb{N}}$ by
$$\psi_n(s)=\mathcal{G}(0,s)\frac{\mathfrak{F}_{n,\tau,\eta}(s)-\mathcal{O}_{\tau,\eta}(s)\delta_n}{|\delta_n|},$$
which by \textbf{(c4)} verifies
$$\lim_{n\to\infty}\psi_n(j)=0.$$

On the other hand, it is easy to deduce from (\ref{221}) and (\ref{222}) that 
\begin{eqnarray*}
    \norm{\frac{\partial f}{\partial u}(s,u)-\frac{\partial f}{\partial u}(s,\tilde{u})}&\leq&\mathfrak{V}(s)|u-\tilde{u}|\text{ , for every }s\in \mathbb{R}^+
\end{eqnarray*}
and
\begin{eqnarray*}
    \norm{\frac{\partial y}{\partial \eta}(s,\tau,\eta)-\frac{\partial y}{\partial \eta}(s,\tau,\tilde{\eta})}&\leq&\pi_\tau(s)|\eta-\tilde{\eta}|\text{ , for }s\geq \tau.
\end{eqnarray*}

Thus, for $s\geq \tau$
\begin{eqnarray*}
    \norm{\mathcal{O}_{\tau,\eta}(s)}&\leq &\norm{\frac{\partial^2 y}{\partial \eta^2}(s,\tau,\eta)}\norm{\frac{\partial  f}{\partial u}(s,y(s,\tau,\eta))}+\norm{\frac{\partial ^2 f}{\partial u^2}(s,y(s,\tau,\eta))} \norm{\frac{\partial y}{\partial \eta}(s,\tau,\eta)}^2\\
    \\&\leq&\pi_\tau(s)\mathfrak{v}(s)+\mathfrak{V}(s)\left[\exp\left(\int_\tau^s \norm{A(r)}+\mathfrak{v}(r)dr\right)\right]^2,
\end{eqnarray*}
and
\begin{eqnarray*}
   \norm{\mathfrak{F}_{n,\tau,\eta}(s)}&\leq&\norm{\frac{\partial f}{\partial u}(s,y(s,\tau,\eta+\delta_n))\frac{\partial y}{\partial \eta}(s,\tau,\eta+\delta_n)-\frac{\partial f}{\partial u}(s,y(s,\tau,\eta+\delta_n))\frac{\partial y}{\partial \eta}(s,\tau,\eta)}\\
    \\&&+\norm{\frac{\partial f}{\partial u}(s,y(s,\tau,\eta+\delta_n))\frac{\partial y}{\partial \eta}(s,\tau,\eta)-\frac{\partial f}{\partial u}(s,y(s,\tau,\eta))\frac{\partial y}{\partial \eta}(s,\tau,\eta)}\\
    \\&\leq&\norm{\frac{\partial f}{\partial u}(s,y(s,\tau,\eta+\delta_n))}\norm{\frac{\partial y}{\partial \eta}(s,\tau,\eta+\delta_n)-\frac{\partial y}{\partial \eta}(s,\tau,\eta)}\\
    \\&&+\norm{\frac{\partial f}{\partial u}(s,y(s,\tau,\eta+\delta_n))-\frac{\partial f}{\partial u}(s,y(s,\tau,\eta))}\norm{\frac{\partial y}{\partial \eta}(s,\tau,\eta)}\\
    \\&\leq&\mathfrak{v}(s)\pi_\tau(s)|\delta_n|+\mathfrak{V}(s)|y(s,\tau,\eta+\delta_n)-y(s,\tau,\eta)|\exp\left(\int_\tau^s \norm{A(r)}+\mathfrak{v}(r)dr\right)\\
    \\&\leq&\left(\mathfrak{v}(s)\pi_\tau(s)+\mathfrak{V}(s)\left[\exp\left(\int_\tau^s \norm{A(r)}+\mathfrak{v}(r)dr\right)\right]^2 \right)|\delta_n|.
\end{eqnarray*}

Now, consider the sequence of continuous functions $s \mapsto\dfrac{\norm{\frac{\partial y}{\partial \eta}(s,\tau,\eta)-\frac{\partial y}{\partial \eta}(s,\tau,\eta+\delta_n)}}{|\delta_n|} $ defined on $[0,\tau]$. As they converge pointwise to the  continuous function $s\mapsto \norm{\frac{\partial^2 y}{\partial \eta^2}(s,\tau,\eta)}$ on the compact domain $[0,\tau]$ the convergence is uniform. Thus, there is $\hat{n}\in \mathbb{N}$ such that
$$\dfrac{\norm{\frac{\partial y}{\partial \eta}(s,\tau,\eta)-\frac{\partial y}{\partial \eta}(s,\tau,\eta+\delta_n)}}{|\delta_n|}\leq\norm{\frac{\partial^2 y}{\partial \eta^2}(s,\tau,\eta)}+2 \text{ for every }0\leq s\leq\tau; n\geq \hat{n}.$$

Thus, for $0\leq s\leq \tau$ and $n\geq \hat{n}$
\begin{eqnarray*}
   \norm{\mathfrak{F}_{n,\tau,\eta}(s)}&\leq&\left(\mathfrak{v}(s)\left(\norm{\frac{\partial^2 y}{\partial \eta^2}(s,\tau,\eta)}+2 \right)+\mathfrak{V}(s)\norm{\frac{\partial y}{\partial \eta}(s,\tau,\eta)}^2\right)|\delta_n|,
\end{eqnarray*}
while for $s\in [0,\tau]$
\begin{eqnarray*}
    \norm{\mathcal{O}_{\tau,\eta}(s)}&\leq& \norm{\frac{\partial^2 y}{\partial \eta^2}(s,\tau,\eta)}\mathfrak{v}(s)+\mathfrak{V}(s)\norm{ \frac{\partial y}{\partial \eta}(s,\tau,\eta)}^2.
\end{eqnarray*}

Let us define
 \begin{equation*}
     \mathscr{F}_\tau(s):=\left\{ \begin{array}{lc}
            2\norm{\mathcal{G}(0,s)}\left(\mathfrak{v}(s)\left(\norm{\frac{\partial^2 y}{\partial \eta^2}(s,\tau,\eta)}+1\right)+\mathfrak{V}(s)\norm{ \frac{\partial y}{\partial \eta}(s,\tau,\eta)}^2\right) & ,0\leq s\leq \tau \\
            \\ 2\norm{\mathcal{G}(0,s)}\left(\mathfrak{v}(s)\pi_\tau(s)+\mathfrak{V}(s)\left[\exp\left(\int_\tau^s \norm{A(r)}+\mathfrak{v}(r)dr\right)\right]^2\right) & ,s\geq \tau.
             \end{array}
   \right.
 \end{equation*}

It is easy to see $\norm{\psi_n(s)}\leq \mathscr{F}_\tau(s)$ for every  $s\in \mathbb{R}^+$ and $n\geq \hat{n}$. Now, using (\ref{223}) we have
\begin{eqnarray*}
\int_{0}^{\infty}\mathscr{F}_\tau(s)ds&\leq& 2\int_{0}^{\tau}\norm{\mathcal{G}(0,s)}\left(\mathfrak{v}(s)\left(\norm{\frac{\partial^2 y}{\partial \eta^2}(s,\tau,\eta)}+1\right)+\mathfrak{V}(s)\norm{ \frac{\partial y}{\partial \eta}(s,\tau,\eta)}^2\right)ds\\ 
&&+2\int_{\tau}^{\infty}K(s)h(s)\left(\mathfrak{v}(s)\pi_\tau(s)+\mathfrak{V}(s)\left[\exp\left(\int_\tau^s \norm{A(r)}+\mathfrak{v}(r)dr\right)\right]^2\right)ds<+\infty.
\end{eqnarray*}

Finally, using Lebesgue's dominated convergence Theorem we obtain
\begin{eqnarray*}
\lim_{n\to \infty}\frac{ \frac{\partial w^*}{\partial \eta}(0;(\tau,\eta+\delta_n))-\frac{\partial w^*}{\partial \eta}(0;(\tau,\eta))+\left[ \int_{0}^\infty \mathcal{G}(0,s)\mathcal{O}_{\tau,\eta}(s)ds\right]\delta_n}{|\delta_n|} &=&-
\lim_{n\to \infty}\int_{0}^\infty\psi_n(s)ds\\
&=&-\int_{0}^\infty \left(\lim_{n\to\infty}\psi_n(s)\right)=0,
\end{eqnarray*}
which implies $\eta\mapsto \frac{\partial w^*}{\partial \eta}(0;(\tau,\eta))$ is differentiable, hence $\eta\mapsto w^*(0;(\tau,\eta))$ is a $C^2$ map.
\end{proof}

\begin{corollary}
If conditions from Lemma hold \ref{224}, then for every fixed $\tau\in \mathbb{R}^+$ 
$$ \frac{\partial^2 w^*(0;(\tau,\eta))}{\partial \eta^2}=-\int_{0}^\infty \mathcal{G}(0,s)\left[\frac{\partial f}{\partial u}(s,y(s,\tau,\eta))\frac{\partial^2 y}{\partial \eta^2}(s,\tau,\eta)+\frac{\partial ^2 f}{\partial u^2}(s,y(s,\tau,\eta))\left( \frac{\partial y}{\partial \eta}(s,\tau,\eta)\right)^2\right]ds.$$
\end{corollary}

\begin{theorem}\label{235}
If conditions \textbf{(c1)-(c5)} hold, with $r=2$ on \textbf{(c4)}, and conditions from Lemma  \ref{224} are satisfied, then (\ref{89}) and (\ref{90}) are $C^2$-topologically equivalent on $\mathbb{R}^+$.
\end{theorem}

Note that Theorem \ref{235} follows easily in the same fashion as the proof of Theorem \ref{225}. Now we proceed to study some technical results in order to establish a concrete example of above Theorem.

\begin{lemma}\label{232}
If conditions \textbf{(c1)-(c4)} hold, with $r=2$ on \textbf{(c4)}, then
$$\frac{\partial}{\partial s}\frac{\partial y}{\partial \eta}(s,\tau,\eta)=\frac{\partial}{\partial \eta}\frac{\partial y}{\partial s}(s,\tau,\eta).$$
\end{lemma}

\begin{proof}
Note that
\begin{eqnarray*}
    \frac{\partial}{\partial s}\frac{\partial y}{\partial \eta}(s,\tau,\eta)&=&\frac{\partial}{\partial s}\left( X(s,\tau)+\frac{\partial }{\partial \eta}\int_\tau^s X(s,r) f(r,y(r,\tau,\eta))dr\right)\\
    \\&=&\frac{\partial}{\partial s}\left( X(s,\tau)+\int_\tau^s X(s,r)\frac{\partial f}{\partial u}(r,y(r,\tau,\eta))\frac{\partial y}{\partial \eta}(r,\tau,\eta)dr\right)\\
    \\&=&A(s)\left(X(s,\tau)+\int_\tau^s X(s,r)\frac{\partial f}{\partial u}(r,y(r,\tau,\eta))\frac{\partial y}{\partial \eta}(r,\tau,\eta)dr\right)+\frac{\partial f}{\partial u}(s,y(s,\tau,\eta))\frac{\partial y}{\partial \eta}(s,\tau,\eta)\\
    \\&=&A(s)\frac{\partial y}{\partial \eta}(s, \tau,\eta)+\frac{\partial f}{\partial u}(s,y(s,\tau,\eta))\frac{\partial y}{\partial \eta}(s,\tau,\eta)\\
    \\&=&\frac{\partial}{\partial \eta}\frac{\partial y}{\partial s}(s,\tau,\eta),
\end{eqnarray*}
where the second equality is verified using Lebesgue's dominated convergence Theorem.
\end{proof}

\begin{corollary}\label{231}
If conditions \textbf{(c1)-(c5)} hold, with $r=2$ on \textbf{(c4)}, then $t\mapsto z(t,\tau,\eta):=\frac{\partial y}{\partial \eta}(t,\tau,\eta)$ is solution to the matrix initial value problem:
\begin{equation*}\label{230}
    \left\{ \begin{array}{ccl}
            z'(t) & = & \left[A(t)+\dfrac{\partial f}{\partial u}(t,y(t,\tau,\eta)\right]z(t)\\
            \\z(\tau) & = & I.
\end{array}
\right.\end{equation*}
\end{corollary}

\begin{lemma}\label{243} Suppose conditions \textbf{(c1)-(c5)} hold, with $r=2$ on \textbf{(c4)}. If $s\mapsto \mathfrak{V}(s)$ satisfies (\ref{221}), then for every $s\geq \tau$
$$\norm{\frac{\partial^2 y}{\partial \eta^2}(s,\tau,\eta)}\leq \exp\left(\int_\tau^s\norm{A(p)}+\mathfrak{v}(p)dp\right)\cdot \int_\tau^s \left\{ \mathfrak{V}(p)\exp \left(2 \int_\tau^p \norm{A(r)}+\mathfrak{v}(r) dr\right) \right\}dp. $$
\end{lemma}

\begin{proof}
Let $\eta,\tilde{\eta}\in \mathbb{R}^n$ and fix $\tau\in \mathbb{R}^+$. Denote $s\mapsto z(s,\tau,\eta):=\frac{\partial y}{\partial \eta}(s,\tau,\eta)$ and $s\mapsto z(s,\tau,\tilde{\eta}):=\frac{\partial y}{\partial \eta}(s,\tau,\tilde{\eta})$. By Corollary \ref{231} we know
$$z'(s,\tau,\eta)=\left[A(s)+\frac{\partial f}{\partial u}(s,y(s,\tau,\eta)) \right]z(s,\tau,\eta),$$
hence
\begin{eqnarray*}
    z'(s,\tau,\eta)-z'(s,\tau,\tilde{\eta})&=&A(s)\left( z(s,\tau,\eta)-z(s,\tau,\tilde{\eta})\right)+\frac{\partial f}{\partial u}(s,y(s,\tau,\eta))z(s,\tau,\eta)\\
    \\& &-\frac{\partial f}{\partial u}(s,y(s,\tau,\tilde{\eta}))z(s,\tau,\tilde{\eta})\\
    \\&=&\left[ A(s)+\frac{\partial f}{\partial u}(s,y(s,\tau,\eta))\right]\left( z(s,\tau,\eta)-z(s,\tau,\tilde{\eta})\right)\\
    \\& &+\left[\frac{\partial f}{\partial u}(s,y(s,\tau,\eta))-\frac{\partial f}{\partial u}(s,y(s,\tau,\tilde{\eta}))\right]z(s,\tau,\tilde{\eta}).
\end{eqnarray*}

Thus
\begin{eqnarray*}
    \norm{z'(s,\tau,\eta)-z'(s,\tau,\tilde{\eta})}&\leq& \norm{A(s)+\frac{\partial f}{\partial u}(s,y(s,\tau,\eta)}\norm{z(s,\tau,\eta)-z(s,\tau,\tilde{\eta})}\\
    \\& &+\norm{\frac{\partial f}{\partial u}(s,y(s,\tau,\eta))-\frac{\partial f}{\partial u}(s,y(s,\tau,\tilde{\eta}))}\norm{z(s,\tau,\tilde{\eta})}\\
    &\leq&\left( \norm{A(s)}+\mathfrak{v}(s)\right) \norm{ z(s,\tau,\eta)-z(s,\tau,\tilde{\eta})} \\
    \\& &+\mathfrak{V}(s)\left|y(s,\tau,\eta)-y(s,\tau,\tilde{\eta})\right|\norm{z(s,\tau,\tilde{\eta})}\\
    \\&\leq&\left( \norm{A(s)}+\mathfrak{v}(s)\right) \norm{ z(s,\tau,\eta)-z(s,\tau,\tilde{\eta})} \\
    \\& &+\mathfrak{V}(s)\exp \left(2 \int_\tau^s \norm{A(r)}+\mathfrak{v}(r) dr \right)|\eta-\tilde{\eta}|,
\end{eqnarray*}
or equivalently, for $p\in [\tau,s]$
\begin{eqnarray*}
\frac{\norm{z'(p,\tau,\eta)-z'(p,\tau,\tilde{\eta})}}{|\eta-\tilde{\eta}|}&\leq& \left( \norm{A(p)}+\mathfrak{v}(p)\right)\frac{\norm{z(p,\tau,\eta)-z(p,\tau,\tilde{\eta})}}{|\eta-\tilde{\eta}|}+\mathfrak{V}(p)\exp \left(2 \int_\tau^p \norm{A(r)}+\mathfrak{v}(r) dr \right).
\end{eqnarray*}

On the other hand
\begin{eqnarray*}
    z(s,\tau,\eta)- z(s,\tau,\tilde{\eta})&=&z(\tau,\tau,\eta)+\int_\tau^s z'(p,\tau,\eta)dp-z(\tau,\tau,\tilde{\eta})-\int_\tau^s z'(p,\tau,\tilde{\eta})dp\\
    \\&=&\int_\tau^s z'(p,\tau,\eta)-z'(p,\tau,\tilde{\eta})dp,
\end{eqnarray*}
hence
$$\phi(s):=\frac{\norm{z(s,\tau,\eta)-z(s,\tau,\tilde{\eta})}}{|\eta-\tilde{\eta}|}\leq \int_\tau^s \frac{\norm{z'(p,\tau,\eta)-z'(p,\tau,\tilde{\eta})}}{|\eta-\tilde{\eta}|}dp,$$
thus, we can deduce
\begin{eqnarray*}
\phi(s)&\leq&\int_{\tau}^s\left\{ \left( \norm{A(p)}+\mathfrak{v}(p)\right)\phi(p)+\mathfrak{V}(p)\exp \left(2 \int_\tau^p \norm{A(r)}+\mathfrak{v}(r) dr \right)\right\}dp,
\end{eqnarray*}
or alternatively, writing for  $s\geq \tau$
$$\alpha(s)=\norm{A(s)}+\mathfrak{v}(s)$$ 
and
$$\beta(s)=\int_\tau^s\mathfrak{V}(p)\exp \left(2 \int_\tau^p \norm{A(r)}+\mathfrak{v}(r) dr\right)dp,$$
we have
$$\phi(s)\leq \beta(s)+ \int_\tau^s  \alpha(p)\phi(p)dp$$
by using Gronwall's Lemma we have
\begin{eqnarray*}
    \phi(s)&\leq& \beta(s)\exp\left(\int_\tau^s\alpha(p)dp\right)\\
    &=& \int_\tau^s \left\{ \mathfrak{V}(p)\exp \left(2 \int_\tau^p \norm{A(r)}+\mathfrak{v}(r) dr\right) \right\}dp\cdot \exp\left(\int_\tau^s\norm{A(p)}+\mathfrak{v}(p)dp\right),
    \end{eqnarray*}
which by definition of $\norm{\frac{\partial^2 y}{\partial \eta^2}(s,\tau,\eta)}$, concludes the proof.
\end{proof}

\begin{corollary} \label{245}
Suppose conditions \textbf{(c1)-(c5)} hold, with $r=2$ on \textbf{(c4)}. If $\mathfrak{v}(s)=\nu e^{-\varepsilon_1 s}$ and $\mathfrak{V}(s)=\zeta e^{-\varepsilon_2 s}$ satisfies (\ref{221}), then for $s\geq \tau$
$$\norm{\frac{\partial^2 y}{\partial \eta^2}(s,\tau,\eta)}\leq \frac{\zeta e^{-3\tau (M+\nu)}}{2(M+\nu)-\varepsilon_2}\left[ e^{(3(M+\nu)-\varepsilon_2)s}-e^{(2(M+\nu)-\varepsilon_2)\tau+(M+\nu)s}\right]. $$
\end{corollary}

\begin{proof} 
By Lemma \ref{243} we have
    \begin{eqnarray*}
    \norm{\frac{\partial^2 y}{\partial \eta^2}(s,\tau,\eta)}&\leq& \exp\left(\int_\tau^s\norm{A(p)}+\mathfrak{v}(p)dp\right)\cdot \int_\tau^s \left\{ \mathfrak{V}(p)\exp \left(2 \int_\tau^p \norm{A(r)}+\mathfrak{v}(r) dr\right) \right\}dp\\
    &\leq&  e^{(s-\tau)(M+\nu)}\int_\tau^s \left\{\zeta e^{-\varepsilon_2 p}e^{2(p-\tau)(M+\nu)}\right\}dp\\
    &=&\zeta e^{(s-3\tau )(M+\nu)} \frac{e^{(2(M+\nu)-\varepsilon_2)s}-e^{(2(M+\nu)-\varepsilon_2)\tau}}{2(M+\nu)-\varepsilon_2}\\
    &=&\frac{\zeta e^{-3\tau (M+\nu)}}{2(M+\nu)-\varepsilon_2}\left[ e^{(3(M+\nu)-\varepsilon_2)s}-e^{(2(M+\nu)-\varepsilon_2)\tau+(M+\nu)s}\right]
\end{eqnarray*}

\end{proof}

\begin{theorem}
Suppose (\ref{89}) admits a nonuniform exponential dichotomy, \textit{i.e. }there are two complementary invariant projectors $P(\cdot)$, $Q(\cdot)$ and constants  $C,\lambda, \varepsilon_1>0$ such that
$$\left\{ \begin{array}{lc}
             \norm{X(t,s)P(s)}\leq C
            e^{-\lambda(t-s)+\varepsilon_1 s}, &\forall t \geq s\geq 0 \\
            \\ \norm{X(t,s)Q(s)}\leq C
            e^{\lambda(t-s)+\varepsilon_1 s}, &\forall 0\leq t\leq s.
             \end{array}
   \right.$$
   
Furthermore, suppose that for each $t\in \mathbb{R}^+$ $u\mapsto f(t,u)$ is a $C^2$ map satisfying
\begin{equation*}
    |f(s,u)|\leq \kappa e^{-\varepsilon_0 s},
\end{equation*}
\begin{equation*}
    \norm{\frac{\partial f}{\partial u}(s,u)}\leq \nu e^{-\varepsilon_1 s}
\end{equation*}
and
\begin{equation}\label{229}
    \norm{\frac{\partial f}{\partial u}(s,u)-\frac{\partial f}{\partial u}(s,\tilde{u})}\leq \zeta e^{-\varepsilon_2 s}|u-\tilde{u}|,
\end{equation}
for given $\kappa,\nu ,\zeta>0$ and $\varepsilon_0 > \varepsilon_1 - \lambda$. If $3M<\lambda+\varepsilon_2$, $2M<\lambda+\varepsilon_2-\varepsilon_1$ and $M<\lambda$, then for small enough $\nu>0$ the systems (\ref{89}) and (\ref{90}) are $C^2$-topologically equivalent on $\mathbb{R}^+.$
\end{theorem}

\begin{proof} By Corollary \ref{228} we know conditions \textbf{(c1)-(c5)} are verified, this the systems are $C^1$-topologically equivalent on $\mathbb{R}^+$, with $\mathfrak{v}(s)=\nu e^{-\varepsilon_1 s}$. It is easy to see (\ref{229}) implies (\ref{221}), with $\mathfrak{V}(s)=\zeta e^{-\varepsilon_2 s}$. As in Corollary \ref{228}, denote
$$\Psi_\tau(s)=\exp\left(\int_\tau^s \norm{A(r)}+\mathfrak{v}(r) dr \right),$$
and note that $s\geq \tau \geq t$ implies $\Psi_t(s)\leq \Psi_\tau(s)$, thus for every fixed   $\tau\in \mathbb{R}^+$
$$\Psi_\tau(s)^2\leq \Psi_0(s)^2\leq e^{2(M+\nu)s}.$$

Hence
\begin{eqnarray*}
     \int_{\tau}^{\infty}K(s)h(s)\mathfrak{V}(s)\Psi_\tau(s)^2 ds &\leq& \int_{0}^\infty K(s)h(s)\mathfrak{V}(s)\Psi_0(s)^2 ds\\
    \\
    &\leq&C\zeta \int_{0}^\infty e^{[-\lambda-\varepsilon_2+\varepsilon_1+2(M+\nu)]s}ds.
\end{eqnarray*}

As $2M<\lambda+\varepsilon_2-\varepsilon_1$, then for a small enough $\nu$ we have
 \begin{equation}\label{233}
    \int_{\tau}^{\infty}K(s)h(s)\mathfrak{V}(s)\left[ \exp\left( \int_\tau^s \norm{A(r)}+\mathfrak{v}(r) dr\right)\right]^2 ds<+\infty.
    \end{equation}

Now, by Corollary \ref{245}, we know
$$\pi_\tau(s):=\frac{\zeta e^{-3\tau (M+\nu)}}{2(M+\nu)-\varepsilon_2}\left[ e^{(3(M+\nu)-\varepsilon_2)s}-e^{(2(M+\nu)-\varepsilon_2)\tau+(M+\nu)s}\right],$$
satisfies condition (\ref{222}) from Lemma \ref{224}. Note that
\begin{eqnarray*}
    \int_{\tau}^{\infty}K(s)h(s)\pi_\tau(s)\mathfrak{v}(s)ds&=&\int_\tau^\infty \mathfrak{K}_\tau \left[ e^{(3(M+\nu)-\varepsilon_2-\lambda)s}-\frac{e^{(M+\nu-\lambda)s}}{e^{(\varepsilon_2-2(M+\nu))\tau}}\right]ds<+\infty,
\end{eqnarray*}
for a small enough $\nu$, where $\mathfrak{K}_\tau=\frac{C\zeta\nu  e^{-3\tau (M+\nu)}}{2(M+\nu)-\varepsilon_2}$. The previous argument, along with (\ref{233}) imply that conditions (\ref{223}) from Lemma \ref{224} is satisfied. Finally, applying Theorem \ref{235} the result follows.
\end{proof}

\begin{remark}
In the previous result, if $\varepsilon_0=\varepsilon_1=\varepsilon_2:=\varepsilon$, then the conditions $\varepsilon_0 > \varepsilon_1 - \lambda$, $3M<\lambda+\varepsilon_2$, $2M<\lambda+\varepsilon_2-\varepsilon_1$ and $M<\lambda$, may be reduced to $2M<\lambda$ and $3M<\lambda+\varepsilon$.
\end{remark}

\end{document}